\documentclass[11pt,a4paper]{article}
\usepackage[margin=2cm]{geometry}
\usepackage{xcolor}
\usepackage{amssymb}
\usepackage{amsmath}
\usepackage{amsthm}
\usepackage[mathscr]{euscript}
\usepackage{bm}
\usepackage{cite}
\usepackage{tensor}
\usepackage[normalem]{ulem}

\DeclareFontShape{OT1}{cmtt}{bx}{n}{<5><6><7><8><9><10><10.95><12><14.4><17.28><20.74><24.88>cmttb10}{}

\numberwithin{equation}{section}
\newcommand{\beq}{\begin{equation}}
\newcommand{\eeq}{\end{equation}}
\newtheorem{Thm}{Theorem}[section]

\newtheorem{lemma}[Thm]{Lemma}
\newtheorem{corollary}[Thm]{Corollary}
\newtheorem{theorem}[Thm]{Theorem}
\theoremstyle{definition}
\newtheorem{example}[Thm]{Example}

\newcommand{\bex}{\begin{example}}
\newcommand{\eex}{\end{example}}

\newcommand{\mbu}{\mathbf{u}}

\newcommand{\mbx}{\mathbf{x}}

\newcommand{\mbz}{\mathbf{z}}
\newcommand{\mbD}{\mathbf{D}}
\newcommand{\mbE}{\mathbf{E}}
\newcommand{\mbF}{\mathbf{F}}
\newcommand{\mbI}{\mathbf{I}}
\newcommand{\mbJ}{\mathbf{J}}
\newcommand{\mbK}{\mathbf{K}}
\newcommand{\mbQ}{\mathbf{Q}}
\newcommand{\mcA}{\mathcal{A}}

\newcommand{\mcC}{\mathcal{C}}
\newcommand{\mcD}{\mathcal{D}}
\newcommand{\mcI}{\mathcal{I}}

\newcommand{\mcQ}{\mathcal{Q}}

\newcommand{\msD}{\mathscr{D}}
\newcommand{\msL}{\mathscr{L}}

\newcommand{\mblam}{\bm{\lambda}}
\newcommand{\p}[0]{\partial}
\newcommand{\gtt}{\textup{\texttt{g}}}
\newcommand{\gttb}{\textup{\textbf{\texttt{g}}}}
\newcommand{\upd}{\mathrm{d}}
\newcommand{\bdag}{\boldsymbol{\dag}}
\newcommand{\mccA}{\tensor*[^{\mathrm{c}\!}]{{\mathcal{A}}}{}}

\newcommand{\uaJ}{u^\alpha_{\mbJ}}

\newcommand{\R}[0]{\mathbb{R}}

\begin{document}
\title{Conservation laws that depend on functions and PDE reduction: extending Noether $1\tfrac{1}{2}$}
\author{Peter E. Hydon\thanks{P.E.Hydon@kent.ac.uk}
        \and John R. King\thanks{John.King@nottingham.ac.uk}
}
\date{}

%
%
%

\maketitle

\begin{abstract}
This paper develops methods for simplifying systems of partial differential equations that have families of conservation laws which depend on functions of the independent or dependent variables. In some cases, such methods can be combined with reduction using families of symmetries, giving a multiple reduction that is analogous to the double reduction of order for ordinary differential equations with variational symmetries. Applications are given, including a widely-used class of pseudoparabolic equations and several mean curvature equations.   

\end{abstract}

\section{Introduction}

Many interesting systems of partial differential equations (PDEs) have Lie pseudogroups of symmetries that depend on arbitrary functions of one independent variable. In 1982, Ovsiannikov \cite{ovsiannikov} showed that these symmetries can be factored out to reduce the PDE system to a simpler form involving only differential invariants, a technique called group splitting. If the simpler system can be solved, so can the original system, because the action of the symmetry pseudogroup on each solution of the simplified system yields a family of solutions of the original system. This approach has been applied widely, and has recently been made completely algorithmic by Thompson \& Valiquette \cite{thomval} using moving frames.

By contrast, relatively little is known about simplifying PDE systems by using families of conservation laws that depend on arbitrary functions of some independent or dependent variables. One indication that this might be fruitful is the following well-known reduction of an Euler--Lagrange ordinary differential equation (ODE), which uses a first integral. Given a Lagrangian functional of the form
\[
\mathscr{L}=\int{L(x,u',u'')}\,\upd x,
\]  
the Euler--Lagrange equation is the fourth-order ODE
\[
\omega(x,u',u'',u''',u''''):= \frac{\upd^2 }{\upd x^2}\left(\frac{\p L}{\p u''}\right) -\,\frac{\upd }{\upd x}\left(\frac{\p L}{\p u'}\right)=0.
\]
This ODE is invariant under all translations in $u$, so a standard symmetry reduction would yield a third-order ODE and a quadrature:
\[
\omega(x,v,v',v'',v''')=0,\qquad u=\int v\, \upd x + c,
\]
where $c$ is an arbitrary constant\footnote{We use $c$ and $c_i$ to denote arbitrary constants from here on.}. However, one can do better than this. The symmetry is variational, so Noether's Theorem applies, giving the first integral
\[
\phi(x,u',u'',u'''):= \frac{\upd }{\upd x}\left(\frac{\p L}{\p u''}\right)-\,\frac{\p L}{\p u'}=c_1.
\]
This inherits the translation symmetries, giving a double reduction to a family of second-order ODEs and a quadrature:
\[
\phi(x,v,v',v'')=c_1,\qquad u=\int v\, \upd x + c_2
\]
The double reduction contains all solutions of the original ODE. If the second-order ODE family can be solved, the general solution of the Euler--Lagrange equation is obtained by quadrature. Even if the solution can be found only for one value of $c_1$ (commonly $c_1=0$), this yields a three-parameter (typically singular) family of solutions of the fourth-order Euler--Lagrange equation.

Sj\"{o}berg \cite{sjo} has partly extended double reduction to non-variational PDEs in two independent variables. Sj\"{o}berg's approach restricts attention to those solutions that are invariant under a one-parameter Lie group of point symmetries which is compatible with a given conservation law. Once such a symmetry group has been found, the conservation law yields an invariant first integral of the ODE that determines the group-invariant solutions. Recently, Anco \& Gandarias \cite{ancogand} strengthened this approach by developing methods of finding all conservation laws that are compatible with a given finite-dimensional symmetry group of the original PDE, without restricting the number of independent variables. In some instances, this produces a complete reduction to quadrature, so that the group-invariant solutions can be found explicitly.

The current paper describes conservation-law reductions that apply to all solutions of a given PDE system, not just group-invariant solutions. After a brief outline of some basic theory (Section 2), we develop the theory of reduction using conservation laws that depend on arbitrary functions of some independent variables (Section 3) or dependent variables (Section 4). We find that, unlike in the ODE double reduction above, symmetries that depend on arbitrary functions are not necessarily symmetries of the reduced PDE. Instead, they act as equivalence transformations that split the reduced PDE into a few inequivalent cases. Commonly, one case inherits symmetries of the original PDE, making at least one further reduction possible by group splitting (Section 5). The resulting methods are applied to a large class of pseudoparabolic equations (Section 6) and to some other well-known PDEs (Section 7), including cases involving an arbitrary function of the dependent variable.

\section{Preliminaries}

\subsection{Orthonomic systems of PDEs}

We consider a given real analytic system of partial differential equations (PDEs) on $\R^N$ in orthonomic form (see below). Independent variables $\mathbf{x} = (x^1,\dots, x^N)$ and dependent variables $\mathbf{u}=(u^1,\dots,u^M)$ are used to state general results; the Einstein summation convention applies throughout. For particular examples, however, commonly-used notation is adopted where this aids clarity.

Derivatives of each $u^\alpha$ are written as $\uaJ$, where $\mbJ=(j^1,\dots,j^N)$ is a multi-index; each $j^i$ denotes the number of derivatives with respect to $x^i$, so $u^\alpha_{\mathbf{0}}= u^\alpha$. The variables $x^i$ and $\uaJ$ can be regarded as coordinates on the infinite jet space (see Olver \cite{olver}). With this approach, the partial derivative with respect to $x^i$ is replaced by the \textit{total derivative},
\[
D_i=\frac{\p}{\p x^i}+u^\alpha_{\mbJ i}\,\frac{\p}{\p \uaJ}\,,\quad\text{where}\quad \mbJ i=(j^1,\dots,j^{i-1},j^i+1,j^{i+1},\dots, j^N).
\]
(Total derivatives commute with one another.) To keep the notation concise, let
\[
\mbD_\mbJ=D_1^{j^1}D_2^{j^2}\dots D_p^{j^N},\qquad |\mbJ|=j^1+\cdots+j^N.
\]
Note that $\uaJ=\mbD_\mbJ u^\alpha$ is a $|\mbJ|$-th order derivative and $u^\alpha_{\mbJ+\mbK}=\mbD_\mbK\uaJ$.

From here on, we use $[\mbu]$ to represent $\mbu$ and finitely many of its total derivatives; more generally, square brackets around an expression denote the expression and as many of its total derivatives as are needed. All functions of $(\mbx,[\mbu])$ are assumed to be analytic, at least locally.

The variables $\uaJ$ may be ranked using any total order $\preceq$ that satisfies the following positivity conditions:
\[
(i)\,\ u^\alpha\prec \uaJ\,,\quad\mbJ\neq\mathbf{0},\qquad (ii)\,\ u^\beta_{\mbI}\prec \uaJ\Longrightarrow D_{\mbK }u^\beta_{\mbI}\prec D_{\mbK }\uaJ\,.
\]
The \textit{leading term} in a differential expression is the highest-ranked $\uaJ$ in the expression, and the rank of the expression is the rank of its leading term (see Rust \cite{rust}).
A system of $m$ PDEs, denoted $\mcA(\mbx,[\mbu])=0$, is orthonomic if its components are of the form
\beq\label{inveqs}
\mcA_\mu=u^{\alpha_\mu}_{\mbJ_\mu}-\omega_\mu(\mbx,[\mbu]),\qquad \mu =1,\dots, m,
\eeq
subject to the following conditions:
\begin{enumerate}
	\item for each $\mu$, $u^{\alpha_\mu}_{\mbJ_\mu}$ is ranked higher than every $\uaJ$ that is an argument of $\omega_\mu\,$;
	\item whenever $\nu\neq\mu$, the leading term $u^{\alpha_\nu}_{\mbJ_\nu}$ is neither $u^{\alpha_\mu}_{\mbJ_\mu}$ nor a derivative of $u^{\alpha_\mu}_{\mbJ_\mu}$;
	\item no $u^{\alpha_\mu}_{\mbJ_\mu+\mbK}$ is an argument of any $\omega_\nu\,$.
\end{enumerate}
Most systems arising from applications can be written in at least one orthonomic form (commonly several, depending on which ranking is used).

An involutive system has no integrability conditions, and yields a formally well-posed initial-value problem (see Seiler \cite{seiler} for details). Every orthonomic system can be completed to an involutive system, denoted $\mccA(\mbx,[\mbu])=0$, by appending all integrability conditions; Marvan \cite{marvan} gives an algorithm for doing this in a finite number of steps. For simplicity, we restrict attention to systems whose involutive completion is orthonomic, subject to the $m$ lowest-ranked components in $\mccA$ being the $\mcA_\mu$ in \eqref{inveqs}.
The leading terms in $\mccA$ and their derivatives (of all orders) are called \textit{principal derivatives}; all other $\uaJ$ are called \textit{parametric derivatives}. Given arbitrary initial data on the set $\mcI$ of all parametric derivatives, the corresponding power-series solution is constructed by using $\mcA=0$ and its prolongations to determine the values of all principal derivatives\footnote{\label{ftwo}The system $\mcA=0$ need not be involutive, because $\mccA$ can be written in terms of $[\mcA]$ and the parametric derivatives.}.

Consequently, any function $f(\mbx,[\mbu])$ may be written, with a slight abuse of notation, as an equivalent function $f(\mbx,\mcI,[\mcA])$. This change of coordinates on the infinite jet space uses the variables $\mbD_\mbK\mcA_\mu,\ |\mbK|\geq 0$, rather than the principal derivatives. It is extremely useful, as it enables any condition of the form
\[
f(\mbx,[\mbu])=0\quad\text{when}\quad[\mcA=0]
\]
(such as the determining equations for symmetries and conservation laws) to be written instead as an equation:
\[
f|_0:=f(\mbx,\mcI,[0])=0.
\]

Some orthonomic systems have syzygies between the components of $\mcA$;
these are differential relations that become identities when they are written in terms of $(\mbx,\mbu)$. If there are syzygies, the coordinates $\mbD_\mbK\mcA_\mu$ have some redundancy.
One resolution is to remove spare coordinates. For conservation laws and symmetries, however, it is easier to accept such redundancies and take them into account in the calculations.

\subsection{Conservation laws: the basics}

A (scalar-valued) \emph{conservation law} for a system $\mcA=0$ on $\mathbb{R}^N$ is a divergence expression,
\beq\label{cldef}
\mcC(\mbx,[\mbu])= \text{Div}\,\mbF= D_iF^i(\mbx,[\mbu]),
\eeq
that is zero on all solutions of the PDE:
\beq\label{cleq}
\mcC|_0=0.
\eeq
For orthonomic systems, \eqref{cleq} implies that there exist functions $f^{\mu,\mbJ}$ such that
\beq\label{clA}
\mcC=f^{\mu,\mbJ}(\mbx,[\mbu])\,\mbD_\mbJ\mcA_\mu;
\eeq
see Anco \cite{anco} for details in the involutive case (and Footnote \ref{ftwo}). Indeed, many conservation laws have more than one such representation.
A conservation law is \textit{trivial} if either of the following conditions hold:
\begin{enumerate}
  \item all components $F^i$ are zero when $[\mcA=0]$, that is, $F^i|_0=0$;
  \item $\mcC=0$ whether or not $[\mcA=0]$ holds (for instance, when $N=3$ and $\mbF$ is a `total curl').
\end{enumerate}
More generally, $\mcC$ is trivial if and only if it is a linear superposition of these two kinds of trivial conservation laws (see Olver \cite{olver}), in which case there exist $\widehat{F}^i$ such that
\beq\label{cltriv}
\mcC=D_i\widehat{F}^i\quad\text{and}\quad \widehat{F}^i|_0=0.
\eeq
Two conservation laws, $\mcC_1,\mcC_2$, are \textit{equivalent} if $\mcC_1\!\!\:-\mcC_2$ is trivial. Every member of an equivalence class of conservation laws expresses the same information about the set of solutions, so equivalent conservation laws are generally treated as being identical.

A conservation law $\mcC$ is in \textit{characteristic form} if
\[
\mcC=\mcQ^\mu\mcA_\mu,
\]
for some functions $\mcQ^\mu(\mbx,[\mbu])$; the $m$-tuple $\mcQ=(\mcQ^1,\dots,\mcQ^m)$ is called a \textit{multiplier}\footnote{The term \textit{characteristic} is also widely-used, but we use multiplier to prevent confusion with symmetry characteristics.} for $\mcC$. For a given conservation law, one can integrate any of its representations \eqref{clA} by parts to obtain an equivalent conservation law in characteristic form. This yields functions $\widetilde{F}^i$ satisfying
\beq\label{clchar1}
\mcQ^\mu\mcA_\mu=D_i\widetilde{F}^i,\quad\text{where}\quad\mcQ^\mu=(-1)^{|\mbJ|}\,\mbD_\mbJ f^{\mu,\mbJ}(\mbx,[\mbu]).
\eeq
The multiplier $\mcQ$ is trivial if $\mcQ|_0=0$. Two multipliers, $\mcQ_1$ and $\mcQ_2$, are equivalent if $\mcQ_1\!\!\:-\mcQ_2$ is trivial. Multipliers are found by solving an overdetermined system of linear PDEs. The solution depends on $[\mbu]$ and functions of $\mbx$ that are subject to linear constraints which may or may not be solvable; some or all of these functions may be unconstrained.
Anco \cite{anco} includes full details of how to determine multipliers in a comprehensive review of conservation laws.

\subsection{Generalized symmetries}

In 1918, Noether \cite{noether} famously introduced the idea of generalized symmetries of a system of PDEs (see Olver \cite{olvnoe} for a discussion of their significance).
For a system $\mcA=0$ with $M$ components, the $M$-tuple $\mbQ=(Q^1(\mbx,[\mbu]),\dots,Q^M(\mbx,[\mbu]))$ is the \textit{characteristic} of a generalized symmetry if the operator
\[
X=\mbD_{\mathbf{J}}\big(Q^{\alpha}(\mbx,[\mbu])\big) \frac{\p}{\p\, \mbD_{\mathbf{J}}u^{\alpha}}
\]
satisfies the \textit{linearized symmetry condition},
\beq\label{lsc}
X(\mcA_{\alpha})\big|_0 =0.
\eeq
Generalized symmetries are not necessarily associated with a Lie group, but the set of all characteristics for a given system is a Lie algebra under the bracket with components
\[
[\mbQ_1,\mbQ_2]^\alpha=X_1(Q^\alpha_2)-X_2(Q^\alpha_1),\qquad \alpha=1,\dots,M.
\]
In particular, the one-parameter Lie group of point symmetries
\[
(\mbx,\mbu)\longmapsto(\widehat{\mbx}(\mbx,\mbu;\varepsilon),\widehat{\mbu}(\mbx,\mbu;\varepsilon)),
\]
defined by
\[
\frac{\upd \widehat{x}^i}{\upd\varepsilon}\,=\xi^i(\widehat{\mbx},\widehat{\mbu}),\qquad \frac{\upd \widehat{u}^\alpha}{\upd\varepsilon}\,=\eta^\alpha(\widehat{\mbx},\widehat{\mbu}),\qquad \left(\widehat{\mbx}(\mbx,\mbu;0),\widehat{\mbu}(\mbx,\mbu;0)\right)=(\mbx,\mbu),
\]
has the characteristic whose components are
\[
Q^\alpha=\eta^\alpha(\mbx,\mbu)-\xi^i(\mbx,\mbu)u^\alpha_i.
\]

\subsection{Variational symmetries and conservation laws}

The formal adjoint of a differential operator $\mcD$ is the unique differential operator $\mcD^{\bdag}$ such that
\[
F\,\mcD G - \left(\mcD^{\bdag} F\right)G
\]
is a divergence for all smooth functions $F$ and $G$. In particular,
\[
(\mbD_\mbJ)^{\bdag}=(-\mbD)_\mbJ:=(-1)^{|\mbJ|}\,\mbD_\mbJ.
\]
Consequently, the \textit{Euler--Lagrange operator} with respect to any variable $v$ that depends on $\mbx$ is
\[
\mbE_v=(-\mbD)_\mbJ\,\frac{\p }{\p v_{\mbJ}}\,,\qquad v_{\mbJ}:=\mbD_{\mbJ}v,
\]
where the total derivatives $\mbD_\mbJ$ now include all dependent variables, including $v$. It is well-known that a function $f(\mbx,[\mbu])$ is a divergence if and only if
\[
\mbE_{u^\alpha}\{f(\mbx,[\mbu])\}=0,\qquad \alpha=1,\dots, M.
\]
The following useful generalization of this result applies to functions that depend on $\mbx$, $[v]$ and a set, $\mbz$, that consists of other (subsidiary) dependent variables and their derivatives.

\begin{lemma}\label{princel}
	Suppose that no syzygies exist between the variables $(\mbx,\mbz,[v])$. Suppose also that for all $i$, each subsidiary $z^a$ on which $F^i(\mbx,\mbz,[v])$ depends satisfies the condition that $D_iz^a$ is a function of $(\mbx,\mbz)$ only. Then $\mbE_v\!\left(D_iF^i\right)=0$. If $\mbE_{v}\big\{f(\mbx,\mbz,[v])\big\}=0$ then $f(\mbx,\mbz,[v])-f(\mbx,\mbz,[0])$ is a divergence.
\end{lemma}
\begin{proof}
	Suppose that $D_iz^a$ is independent of $[v]$, for all $z^a$ occurring in $F^i$. Then
	\begin{align*}
		\mbE_v\!\left(D_iF^i(\mbx,\mbz,[v])\right)
		&=(-\mbD)_\mbJ\left\{D_i\!\left(\frac{\p F^i}{\p v_\mbJ}\right)\!+\frac{\p D_iz^a}{\p v_\mbJ} \,\frac{\p F^i}{\p z^a}+\,\frac{\p v_{\mbK i}}{\p v_\mbJ} \,\frac{\p F^i}{\p v_\mbK}\right\}\\
		&=(-\mbD)_\mbJ\left\{D_i\!\left(\frac{\p F^i}{\p v_\mbJ}\right)+\,\frac{\p v_{\mbK i}}{\p v_\mbJ} \,\frac{\p F^i}{\p v_\mbK}\right\}\\
		&=0.
	\end{align*}
	(Syzygies invalidate the conclusion `$=0$'.) Now suppose that $\mbE_{v}\big\{f(\mbx,\mbz,[v])\big\}=0$. Then
	\begin{align*}
		f(\mbx,\mbz,[v])-f(\mbx,\mbz,[0])&=\int_{t=0}^1\frac{\upd f(\mbx,\mbz,[tv])}{\upd t}-v\left(\mbE_{v}\big\{f(\mbx,\mbz,[v])\big\} \right)\!\big|_{v\mapsto tv}\,\upd t\\
		&=\int_{t=0}^1\left\{v_\mbJ\,\frac{\p f(\mbx,\mbz,[v])}{\p v_\mbJ}-v(-\mbD)_\mbJ\,\frac{\p f(\mbx,\mbz,[v])}{\p v_\mbJ}\right\}\bigg|_{v\mapsto tv}\frac{\upd t}{t}\,.
	\end{align*}
	The expression in braces is a divergence, so $f(\mbx,\mbz,[v])-f(\mbx,\mbz,[0])$ is a divergence.
\end{proof}

For the rest of this subsection, let $\mcA=0$ be the system of Euler--Lagrange equations,
\beq\label{eleq}
\mcA_\alpha:={\mathbf{E}}_{u^\alpha}\{L(\mbx,[\mbu])\}=0,\qquad \alpha =1,\dots,M,
\eeq
arising from from the variational problem $\delta\msL=0$, where
\[
\msL=\int L(\mbx,[\mbu]) \upd\,\mbx
\]
is the action. We begin with a brief outline of variational symmetries; for more details, see Olver \cite{olver}.

A generalized symmetry is \textit{variational} if, for all sufficiently small $|\varepsilon|$, the mapping
\[
(\mbx,\,\uaJ)\longmapsto (\mbx,\,\uaJ+\varepsilon \mbD_{\mbJ} Q^\alpha)
\]
leaves the action unchanged, to first order in $\varepsilon$. Consequently, the Euler--Lagrange equations are unchanged, to first order. As
\[
L\mapsto L+\varepsilon X(L) + O(\varepsilon^2),
\]
it follows that $X(L)$ belongs to the kernel of all $\mbE_{u^\alpha}$, so it is a divergence. Therefore, $\mbQ$ is a characteristic of generalized variational symmetries if and only if there exist functions $P^i(\mbx,[\mbu])$ such that
\beq\label{vargen}
X(L) = \mbD_{\mbJ}Q^{\alpha}\, \frac{\partial L}{\p\, \mbD_{\mbJ}u^{\alpha}} =D_i P^i.
\eeq
This can be integrated by parts to obtain a conservation law in characteristic form,
\beq\label{varcl}
\qquad Q^{\alpha}{\mathbf{E}}_{\alpha}(L) = D_i F^i(\mbx,[\mbu]),
\eeq
where $F^i-P^i$ is the divergence arising from the integration.
Furthermore, given any conservation law in characteristic form, the argument above can be reversed to go from \eqref{varcl} to \eqref{vargen} (with $\mcQ^\alpha$ replacing $Q^\alpha$). This leads to Noether's (First) Theorem from her 1918 paper.

\begin{theorem}[Noether $1$]
	A function $\mbQ$ is a generalized variational symmetry characteristic for a system of Euler--Lagrange equations if and only if it is a multiplier of a conservation law.
\end{theorem}

Noether's Second Theorem deals with families of characteristics that depend on a completely arbitrary function $\gtt(\mbx)$ and its derivatives.

\begin{theorem}[Noether $2$]
	An Euler--Lagrange system, $\mbE_{u^\alpha}(L)=0$, has a family of variational symmetry characteristics $\mbQ(\mbx,[\mbu],[\gtt])$ that are linear homogeneous in an arbitrary function $\gtt(\mbx)$ if and only if there is a syzygy between the components $\mbE_{u^\alpha}(L)$, that is, if and only if there exist differential operators $\mcD^\alpha$ (independent of $\gtt$) such that
	\beq\label{noesyz}
	\mcD^{\alpha} \mbE_{u^\alpha}(L)= 0.
	\eeq
\end{theorem}

Specifically, the family of characteristics is obtained from this syzygy by multiplying \eqref{noesyz} by $\gtt$, then integrating by parts to obtain an expression of the form \eqref{varcl}, whose right-hand side is (automatically) a trivial conservation law.
Recently, Olver has proved that the arbitrary function $\gtt$ may also depend on $[\mbu]$ (see Olver \cite{olvN2} for details). Noether's Second Theorem immediately generalizes to families of characteristics that depend on more than one arbitrary function, each of which corresponds to an independent syzygy.

\section{Conservation laws that depend on functions of $\mbx$}\label{clfnsec}

From here on, we consider a given PDE system $\mcA=0$ that is not necessarily variational.
The set $C_{\:\!\!\mcA}$ of equivalence classes of conservation laws is a vector space. If the dimension of $C_{\:\!\!\mcA}$ is countable and $\mcC_k$ is a conservation law in the $k^{\text{th}}$ equivalence class, every conservation law is equivalent to $\mcC=c^k\mcC_k$ for some constants $c^k$.  However, many well-known systems have conservation laws in a family $\mcC(\mbx,[\mbu],[\gttb])$ that is linear homogeneous in $\gttb=(\gtt^1,\gtt^2,\dots)$, where each $\gtt^r$ is a smooth function of $\mbx$ only. The functions $\gtt^r$ may be free (entirely arbitrary), or arbitrary subject to some linear differential constraints,
\beq\label{constr}
\msD^l_r\gtt^r=0,\qquad l=1,2,\dots\, .
\eeq
Here each $\msD^l_r$ is a linear differential operator whose coefficients depend on $\mbx$ only; total derivatives $D_i$ are used, even though $\gtt^r$ does not depend on $[\mbu]$. The set of constraints \eqref{constr} is \textit{complete} if:
\begin{itemize}
	\item it fully specifies $\gttb$, so that each $\gtt^r$ is an arbitrary function of $\mbx$ subject only to \eqref{constr};
	\item it has no additional integrability conditions.
\end{itemize}

\noindent\textbf{Note}:\,\ If the dimension of $C_{\:\!\!\mcA}$ is countable, every conservation law $\mcC$ is a member of the family $\mcC=\gtt^k\mcC_k(\mbx,[\mbu])$ that is subject to the constraints $D_i\gtt^k=0$ for all $i$ and $k$.

Multipliers that depend on functions $\gttb$ can be used to obtain syzygies and/or simplify the conservation laws of the given system of PDEs.

\begin{theorem}\label{bridge1}
	Suppose that $\mcA=0$ has a family of conservation laws $\mcC$ that is linear homogeneous in $\gttb$, which is subject to the complete set of constraints \eqref{constr}. Then there exists $\mblam=(\lambda_1(\mbx,[\mbu]),\lambda_2(\mbx,[\mbu]),\dots)$ such that
	\beq\label{detconstr}
	(-\mbD)_\mbJ\left(\frac{\p\, \mcC(\mbx,[\mbu],[\gttb])}{\p \gtt^r_{\mbJ}}\right)+\big(\msD^l_r\big)^{\!\bdag}\lambda_l=0,\qquad r=1,2,\dots.
	\eeq
	For each set of differential operators $\msD^r$ such that $\msD^r(\msD^l_r)^{\!\bdag}\lambda_l=0$, there is a corresponding syzygy,
	\beq\label{syzelim}
	\msD^r(-\mbD)_\mbJ\left(\frac{\p\, \mcC(\mbx,[\mbu],[\gttb])}{\p \gtt^r_{\mbJ}}\right)=0.
	\eeq
	Provided that not all $\gtt^r$ are free, the family $\mcC$ is equivalent to the family of conservation laws obtained by substituting any solution $\mblam$ of \eqref{detconstr} into
	\beq\label{clcon}
	\mcC_{\mblam}:=\lambda_l\msD^l_r\gtt^r-\gtt^r\big(\msD^l_r\big)^{\!\bdag}\lambda_l.
	\eeq
\end{theorem}
\begin{proof}
	Suppose that $\mcA=0$ has such a family of conservation laws. If any of the functions $\gtt^r$ are free, Lemma \ref{princel} applies with $v=\gtt^r$, giving
	\beq\label{enocon}
	\mbE_{\gtt^r}\!\left\{\mcC(\mbx,[\mbu],[\gttb])\right\}=0.
	\eeq
	For all other $r$, use Lagrange multipliers, $\lambda_l$, to enforce the constraints, giving
	\beq\label{econstr}
	\mbE_{\gtt^r}\!\left\{\mcC(\mbx,[\mbu],[\gttb])+\lambda_l\:\!\msD^l_s\gtt^s\right\}=0.
	\eeq

	This amounts to \eqref{detconstr}, from which the syzygies \eqref{syzelim} are constructed. As $\lambda_l\msD_s^l=0$ when $\gtt^s$ is free, \eqref{enocon} can be treated as a special case of the general construction. If not all $\msD^l_r$ are zero, there exists a solution $\mblam(\mbx,[\mbu])$ of \eqref{detconstr}.
	
	The conservation law $\mcC$ differs from its characteristic form $\mcQ^\mu\mcA_\mu$ by a divergence. Lagrange multipliers enable the functions $\gtt^r$ to be treated as if they were unrelated so, by Lemma \ref{princel},
	\beq\label{chardet}
	\mbE_{\gtt^r}\{\mcC+\lambda_l\:\!\msD^l_s\gtt^s\}=\mbE_{\gtt^r}\{\mcQ^\mu\mcA_\mu+\lambda_l\:\!\msD^l_s\gtt^s\}=(-\mbD)_\mbJ\left(\frac{\p\, \mcQ^\mu(\mbx,[\mbu],[\gttb])}{\p \gtt^r_{\mbJ}}\,\mcA_\mu\right)+\big(\msD^l_r\big)^{\!\bdag}\lambda_l.
	\eeq
	As $\mcQ$ is linear homogeneous in $\gttb$,
	\[
	\mcQ^\mu\mcA_\mu-\mcC_{\mblam}=\gtt^r_{\mbJ}\, \frac{\p \mcQ^\mu}{\p \gtt^r_{\mbJ}}\,\mcA_\mu-\gtt^r(-\mbD)_\mbJ\left(\frac{\p \mcQ^\mu}{\p \gtt^r_{\mbJ}}\,\mcA_\mu\right),
	\]
	for all $\gttb$ that satisfy the constraints. This is a trivial conservation law; thus, so is $\mcC-\mcC_{\mblam}$.
	
	Finally, suppose that $\mblam^{1}$ and $\mblam^{2}$ both satisfy \eqref{detconstr}, and let $\mblam=\mblam^{1}-\mblam^{2}$. Then $\big(\msD^l_r\big)^{\!\bdag}\lambda_l=0$, and so $\mcC_{\mblam}=\mcC_{\mblam^{1}}-\mcC_{\mblam^{2}}$ is zero for all $\gttb$ that satisfy the given constraints. Therefore, every solution of \eqref{detconstr} yields a family of conservation laws \eqref{clcon} that is equivalent to $\mcC$.
\end{proof}

It is well-known that if $\ell u=0$ is a linear homogeneous PDE (where $\ell$ is a linear differential operator) and $\gtt(\mbx)$ is a solution of $\ell^{\bdag}\gtt=0$ then
\beq\label{linCL}
\mcC=\gtt \ell u-u\ell^{\bdag}\gtt
\eeq
is a conservation law. This arises from Theorem \ref{bridge1} by setting $\mcA_1=\ell u$ and $\mcQ^1=\gtt^1=\gtt(\mbx)$; the condition $\mbE_u(\mcQ^1\mcA_1)=0$ gives the constraint $\ell^{\bdag}\gtt=0$, so $\msD_1^1=\ell^{\bdag}$ and hence \eqref{detconstr} has the solution $\lambda_1=-u$. Theorem \ref{bridge1} extends this result to nonlinear systems of PDEs.

As Theorem \ref{bridge1} implies, there are systems of PDEs and constraints for which it is possible to eliminate Lagrange multipliers from \eqref{detconstr} and its differential consequences to obtain syzygies.

\bex
In Cartesian coordinates, the Euler equation for a constant-density two-dimensional incompressible potential flow with velocity $\nabla\phi$ and pressure $p$ has the components
\[
\mcA_1=\phi_{xt}+\phi_x\phi_{xx}+\phi_y\phi_{xy}+p_x,\qquad \mcA_2=\phi_{yt}+\phi_x\phi_{xy}+\phi_y\phi_{yy}+p_y,\qquad\mcA_3=\phi_{xx}+\phi_{yy}\,.
\]
This system has an integrability condition\footnote{One orthonomic form of the completed involutive system has leading derivatives $(\phi_{xt},\phi_{yt},\phi_{xx},p_{xx})$, the last of which is the leading term in the integrability condition
\[
0=p_{xx}+p_{yy}+2(\phi_{xy}^2+\phi_{yy}^2)=D_x\mcA_1+D_y\mcA_2+(\phi_{yy}-D_t-\phi_xD_x-\phi_yD_y)\mcA_3-\mcA_3^2.
\]}
that determines the Laplacian of the pressure in terms of $[\phi]$.
All multipliers that are independent of $([\phi],[p])$ are of the form $\mcQ^\mu=\gtt^{\mu}(x,y,t)$, subject to the following constraints:
\[
D_x\gtt^1+D_y\gtt^2=0,\qquad D_x^2\gtt^3+D_y^2\gtt^3=0.
\]
Applying Theorem \ref{bridge1} gives the conditions
\beq\label{potcond}
\mcA_1=D_x\lambda_1,\qquad \mcA_2=D_y\lambda_1,\qquad \mcA_3=-D_x^2\lambda_2-D_y^2\lambda_2,
\eeq
which have a solution $(\lambda_1,\lambda_2)=(H,-\phi)$,
where $H=\phi_t+(\phi_x^2+\phi_y^2)/2+p$ is the Bernoulli function. The corresponding family of conservation laws is $\mcC_{\mblam}=\mcC_1+\mcC_2$, where
\[
\mcC_1=D_x(\gtt^1H)+D_y(\gtt^2H),\qquad \mcC_2=D_x(\gtt^3\phi_x-\gtt^3_x\phi)+D_y(\gtt^3\phi_y-\gtt^3_y\phi).
\]
Note that $\lambda_1$ can be eliminated from \eqref{potcond}, yielding the syzygy
\[
D_x\mcA_2-D_y\mcA_1=0.
\]
The conservation law $\mcC_1$ is trivial, because the solution of the first constraint is $(\gtt^1,\gtt^2)=(\gtt_y,-\gtt_x)$, where the arbitrary function $\gtt(x,y,t)$ can be regarded as the Lagrange multiplier for the syzygy. Consequently,
\[
\mcC_1=D_x\left(D_y(\gtt H)-\gtt\mcA_2\right)+D_y\left(-D_x(\gtt H)+\gtt\mcA_1\right)=D_x\left(-\gtt\mcA_2\right)+D_y\left(\gtt\mcA_1\right).
\]
By contrast, $\mcC_2$ is nontrivial for $\gtt^3\neq 0$; indeed, it is the conservation law \eqref{linCL} for Laplace's equation.
\eex

\begin{corollary}\label{corpart}
	If the constraints on $\gttb$ in Theorem \ref{bridge1} do not involve derivatives with respect to one or more variables $x^i$, the corresponding components $F^i$ in the family of conservation laws \eqref{clcon} are zero.
\end{corollary}
\begin{proof}
	The total differential operators in $\big(\msD^l_r\big)^{\!\bdag}$ have no $D_{x^i}$ derivatives, so neither does \eqref{clcon}.
\end{proof}

In particular, Corollary \ref{corpart} applies whenever $\gttb$ is an arbitrary function of some (but not all) independent variables\footnote{This was established for quasi-Noether systems by Rosenhaus \& Shankar \cite{rosenshank}, and for scalar PDEs by Popovych \& Bihlo \cite{popbih}}. 
The best-known examples with this property are scalar PDEs with first integrals, such as the Liouville equation and other Darboux integrable scalar hyperbolic PDEs. However, the corollary applies equally to multi-component systems whose multipliers depend one or more arbitrary functions of $N-1$ variables.

\bex
Consider the Liouville-type system $\mcA=0$ defined by
\[
\mcA_1=u_{xt}-e^{2u-v},\qquad \mcA_2=v_{xt}-e^{2v-u}.
\]
This has a family of multipliers that depend on $(u,v,u_x,v_x)$ and $\gtt^1(x,t)$, with components
\[
\mcQ^1=\gtt^1(2u_x-v_x)+D_x\gtt^1,\qquad \mcQ^2=\gtt^1(2v_x-u_x)+D_x\gtt^1,
\]
subject to the constraint $D_t\gtt^1=0$. The corresponding conservation law in characteristic form is
\[
\mcC=D_x\left\{-\,\gtt^1\!\left(e^{2u-v}+e^{2v-u}\right)\right\}+D_t\left\{\gtt^1\!\left(u_x^2-u_xv_x+v_x^2\right)+(D_x\gtt^1)(u_x+v_x)\right\}.
\]
The condition \eqref{detconstr} amounts to
\[
(2u_x-v_x-D_x)\mcA_1+(2v_x-u_x-D_x)\mcA_2-D_t\lambda_1=0,
\]
which has a solution
\[
\lambda_1=u_x^2-u_xv_x+v_x^2-u_{xx}-v_{xx}\,.
\]
So this family of multipliers leads to the family of conservation laws $D_t(\gtt^1 \lambda_1)$ for all $\gtt^1$ that depend on $x$ only; in other words, $\lambda_1$ is a first integral.

Similarly, for the multiplier with components
\[
\mcQ^1=\gtt^2(2u_xv_x-v_x^2+v_{xx})+2(D_x\gtt^2)u_x+D_x^2\gtt^2,\qquad \mcQ^2=\gtt^2(u_x^2-2u_xv_x-u_{xx})-(D_x\gtt^2)u_x\,,
\]
subject to $D_t\gtt^2=0$, the condition \eqref{detconstr} leads to another first integral,
\[
\lambda_2=u_x(v_x^2-u_xv_x+2u_{xx}-v_{xx})-u_{xxx}\,.
\]
Unlike $\lambda_1$, the function $\lambda_2$ is not symmetric under the discrete symmetry $(u,v)\mapsto(v,u)$. This symmetry produces the first integral $D_x\lambda_1-\lambda_2$. The discrete symmetry $(x,t)\mapsto(t,x)$ generates further first integrals from $\lambda_1$ and $\lambda_2$.
\eex
More generally, the constraint need not be of the form $D_i\gtt=0$ for a first integral to arise for a PDE system with two independent variables. The constraint $\msD\gtt=0$ is sufficient, for any first-order differential operator $\msD$ whose coefficients depend on $\mbx$ only.

Another application of the results in this section occurs where symbolic algebra has been used to derive conservation laws (see Poole \& Hereman\cite{pooleherm} and Wolf \cite{wolf}, for example); these are not necessarily in their simplest equivalent form. Once a family of conservation laws that depend on a function is known, Theorem \ref{bridge1} may be used with Corollary \ref{corpart} to find an equivalent (perhaps simpler) form.

\bex
For the KP equation, written as the system
\[
v_x-u_{yy}=0,\qquad v=u_{t}+2uu_{x}+u_{xxx}\,,
\]
symbolic methods yield the following family of conservation laws with $\gtt$ a function of $t$ only:
\[
\mcC=D_t\{\gtt yu\}+D_x\left\{\gtt yu_{xx}+\gtt yu^2-\big(\tfrac{1}{6}\gtt_{t}y^3+\gtt xy\big)v\right\}+D_y\left\{\big(\tfrac{1}{6}\gtt_{t}y^3+\gtt xy\big)u_{y}-\tfrac{1}{2}\gtt_{t}y^2u-\gtt xu\right\}\!.
\]
Applying Theorem \ref{bridge1} with the constraints $D_x\gtt=0$, $D_y\gtt=0$, gives the equivalent family of conservation laws,
\[
\mcC_\lambda=D_x\left\{\gtt\big( yu_{xx}+yu^2+\tfrac{1}{6}y^3v_t-xyv\big)\right\}+D_y\left\{\gtt\big(-\tfrac{1}{6}y^3u_{ty}+xyu_{y}+\tfrac{1}{2}y^2u_{t}-xu\big)\right\},
\]
which is a little simpler than $\mcC$, as it has one fewer term and no $D_t$ component. For higher-order conservation laws with many terms, greater simplification can occur.
\eex

\medskip
\noindent\textbf{Note.}\ When the vector space $C_\mcA$ of equivalence classes of conservation laws has countable dimension, with a basis $\{\mcC_k=D_iF_k^i(\mbx,[\mbu]): k=1,2,\dots,\}$), one can represent the set of all conservation laws (up to equivalence) by
\[
\mcC=\gtt^k\mcC_k,\quad\text{subject to}\quad D_i\gtt^k=0.
\]
With a slight variation in notation, let $\lambda^i_k$ be the Lagrange multiplier corresponding to the constraint $D_i\gtt^k=0$. Then \eqref{detconstr} amounts to
\[
\mbE_{\gtt^j}(\mcC)=D_i\lambda^i_j\,,
\]
which has the solution $\lambda^i_j=F^i_j(\mbx,[\mbu])$; the resulting conservation law is $\mcC_{\mblam}=\mcC$. So the approach taken above applies equally whether $C_\mcA$ is finite- or infinite-dimensional.

\section{Hodograph transformations}

This section examines what happens when a system of PDEs has a family of conservation laws that depend on arbitrary functions which involve dependent variables. We restrict attention to systems for which there exists a hodograph transformation that gives all such variables the role of independent variables. Let $\widetilde{x}^i(\mbx,\mbu),\ i=1,\dots N$ be the new independent variables after such a transformation, and let $\widetilde{u}^\alpha(\mbx,\mbu),\ \alpha=1,\dots,M$ be the new dependent variables. For the transformation to be valid, it is necessary that the Jacobian determinant, $\mathcal{J}=\text{det}(D_j \widetilde{x}^i)$, is non-zero. Let $\widetilde{D}_i$ denote the total derivative with respect to $\widetilde{x}^i$ in the new variables. Then the change of variables rule for a total divergence gives the following useful result.

\begin{lemma}
Let $C=D_i(F^i)$ be a conservation law for a given PDE system. Then $\widetilde{C}=\mathcal{J}^{-1}C$ is a conservation law for the same system in the hodograph-transformed variables, so there exist functions $\widetilde{F}^i$ such that $\widetilde{C}=\widetilde{D}_i(\widetilde{F}^i)$.
\end{lemma}

Consequently, the results of Section \ref{clfnsec} can be applied immediately to the transformed PDE system.

\bex
Consider the time-dependent problem of flow by mean curvature via its level-set formulation,
\beq\label{N}
|\nabla u| \nabla\cdot \left(\frac{\nabla u}{|\nabla u|} \right)-u_t=0.
\eeq
We will work in two spatial dimensions, though the following conservation-law reduction generalizes readily to higher dimensions and to other forms of interfacial dynamics. The normal velocity of each level set of $u(x,y,t)$ is given by the negative of its curvature (so that closed curves disappear in finite time). This geometrical content of (\ref{N}), whereby there is no coupling between different level sets, is reflected in its having three obvious families of symmetries,
\beq\label{N+1}
u\longmapsto U(u),\quad U'(u)\neq 0,\qquad x\longmapsto x + X(u),\qquad y \longmapsto y + Y(u),
\eeq
each involving an arbitrary function of $u$. Equation \eqref{N} admits every multiplier of the form $\mcQ=\gtt(u)$; in regions where $u_x\neq 0$, the resulting conservation law is
\beq\label{redtdmc}
D_x\{-\gtt(u)u_y\tan^{-1}(u_y/u_x)\}+D_y\{\gtt(u)u_x\tan^{-1}(u_y/u_x)\}+D_t\left\{-\int\gtt(u)\,\upd  u\right\}=0.
\eeq
Proceeding as above in applying a hodograph transformation with $u$ as an independent variable and $x=x(u,y,t)$ as a dependent variable, (\ref{redtdmc}) reduces to the product of $\gtt(u)$ and the conservation law
\beq
D_y\left\{-\tan^{-1}(x_y)\right\}+D_t\{x\}=0.
\eeq
A similar reduction using $y$ as a dependent variable applies in regions where $u_y\neq 0$. (One could also use $t$ a dependent variable, but the resulting conservation law is not much simpler than \eqref{redtdmc}.)

So conservation-law reduction simplifies the time-dependent two-dimensional mean curvature flow to the nonlinear filtration equation
\beq\label{N+2}
x_t=\frac{x_{yy}}{1+x_y^2}\,;
\eeq
see Ibragimov \cite{ibrag} for symmetries and some exact solutions of \eqref{N+2}. The five Lie point symmetries in Ibragimov \cite{ibrag} amount to translations in $x, y$ and $t$, a scaling, and rotations in the $(x,y)$-plane. Each of these corresponds to a family of symmetries of \eqref{N} that is obtained by replacing each group parameter by an arbitrary function of $u$. For instance, the scaling invariance of (\ref{N+2}) amounts to the family of symmetries
\[
x \longmapsto h(u)x,\qquad y \longmapsto h(u)y,\qquad t\longmapsto h^{2}(u)t,
\]
of (\ref{N}) for arbitrary nonzero $h(u)$, which is perhaps not obvious \emph{a priori} from \eqref{N}\footnote{The symmetries of \eqref{N} are most evident when the equation is written in its Schwarz-function form (see King \cite{king2000}).}. The first family of symmetries in \eqref{N+1} becomes trivial in the reduced equation, as $u$ appears only parametrically in (\ref{N+2}) (so that the dimensionality is lowered).

Equation (\ref{N+2}) is of course familiar in the interfacial dynamics context: parametrizing the level set of interest in the form $x=f(y,t)$ by setting $u=x-f(y,t)$ leads at once to (\ref{N+2}) with $x$ replaced by $-f$; the reduction to (\ref{N+2}) is in this sense obvious. 

Among the many natural generalizations of (\ref{N}), here we mention only the anisotropic case
\[
u_t=\Phi'\left(\tan^{-1}\left(u_y/u_x\right)\right)|\nabla u| \nabla. \left(\frac{\nabla u}{|\nabla u|} \right),
\]
that similarly reduces to
\[
D_y\left\{\Phi(-\tan^{-1}(x_y))\right\}+D_t(x)=0.
\]
\eex

\section{Multiple reduction using conservation laws and symmetries}

For clarity, we now consider scalar PDEs $\mcA=0$ with two independent variables, $(x,t)$.
Suppose that a given PDE has a family of conservation laws whose multipliers depend linearly on an arbitrary function $\gtt(x,t)$ that is subject to a constraint of the form
\[
a(x,t)D_x\gtt+b(x,t)D_t\gtt=0.
\]   
Using the method of characteristics to change variables if necessary, suppose without loss of generality that $D_x\gtt=0$. Corollary \ref{corpart} implies that (up to equivalence) the corresponding family of conservation laws \eqref{clcon} is of the form
\[
\mcC:=D_x\{\gtt(t)\lambda(x,t,[u])\}=0.
\]
Consequently, $\lambda(x,t,[u])$ is a first integral and so the PDE reduces to
\begin{equation}\label{reduced}
\lambda(x,t,[u])=f(t),
\end{equation}
where $f$ is arbitrary. Symmetries of $\mcA=0$ map the set of solutions to itself, so they are equivalence transformations of \eqref{reduced}, that is, they may change $f$. Thus, a known Lie group of symmetries can be used to partition the reduced PDE \eqref{reduced} into equivalence classes. Commonly, the symmetry group depends on at least one arbitrary function, $h(t)$. In this case, typically, the number of equivalence classes is small and each class has a member with $f$ constant. At the other extreme, some nontrivial symmetries of $\mcA=0$ may become trivial symmetries of the reduced equation \eqref{reduced}; such symmetries do not simplify \eqref{reduced}.

\bex\label{expot}

The type of reduction described above has a familiar application, namely, the transformation of an evolutionary PDE in conservation form,
\beq\label{PDEcons}
u_t=D_xF(x,t,[u])
\eeq
to its potential form. To see this, substitute $u=w_x$ into \eqref{PDEcons} to get
\beq\label{PDEraised}
\mcA:=-w_{xt}+D_xF(x,t,[w_x])=0.
\eeq
This PDE has a family of multipliers $\mcQ=\gtt(t)$, which yields the conservation law
\beq
D_x\{-w_{t}+F(x,t,[w_x])\}=0.
\eeq
Thus, there is a first integral,
\beq\label{PDEfi}
-w_{t}+F(x,t,[w_x])=f(t).
\eeq
The PDE \eqref{PDEraised} also has the family of point symmetries $w\mapsto w+h(t)$, which map \eqref{PDEfi} to
\beq
-w_{t}+F(x,t,[w_x])=f(t)+h'(t).
\eeq
By choosing $h$ appropriately, it is clear that there is only one equivalence class, which contains the potential form of \eqref{PDEcons}:
\beq\label{PDEpot}
w_{t}=F(x,t,[w_x]).
\eeq
So the well-known transformation from \eqref{PDEcons} to \eqref{PDEpot} works because there are appropriate families of multipliers and symmetries.
\eex

Suppose that a member of an equivalence class is invariant under a Lie group of nontrivial point symmetries of \eqref{reduced} that depend on one or more arbitrary functions of one variable (whether or not these are symmetries of $\mcA=0$). Then group splitting can be used to reduce \eqref{reduced} further. If the further-reduced PDE can be solved, all solutions in the relevant equivalence class can be reconstructed.

If $\mcA=0$ is an Euler--Lagrange equation, the multiplier depending on $\gtt(t)$ is also a variational symmetry characteristic (by Noether's Theorem). Therefore, partitioning into a small number of equivalence classes is always possible. If an equivalence class is invariant under the symmetries that depend on $\gtt(t)$, group splitting gives a true double reduction for that class which mirrors the double reduction of order for Euler--Lagrange ODEs.

\section{Pseudoparabolic examples}

We illustrate the above general principles through application to pseudoparabolic PDEs of the form
\beq\label{jrk6.1}
\mcA:= -u_t+ D_x\{M(u) D_xD_t\Psi(u) \}=0,\qquad M(u)\Psi'(u)\neq 0.
\eeq
Such equations, typically with an additional diffusion term, arise in a variety of applications (see Barenblatt \textit{et al.} \cite{barenblatt} and Cuesta \& Hulshof \cite{cuestahul}, for example), with the limit case of negligible diffusivity, as in (\ref{jrk6.1}), being of specific interest (e.g. King \cite{king}). The general theory furnishes results for these PDEs that we believe to be new.

We begin with the restriction
\beq\label{jrk6.2}
M(u)=\Psi'(u),
\eeq
which enables \eqref{jrk6.1} be written as an Euler--Lagrange equation by substituting $u=w_x$, as follows:
\beq\label{psfourth}
\mcA:= -w_{xt}+ D_x\{\Psi'(w_x) D_xD_t\Psi(w_x) \}=0,\qquad \Psi'(w_x)\neq 0.
\eeq
The Lagrangian functional is
\beq\label{JRK6.3}
\mathcal{L}=\int\hspace{-0.2cm}\int\tfrac{1}{2}w_xw_t+
\tfrac{1}{2} \left(\Psi'(w_x)\right)^2
w_{xx}w_{xt}\,
\upd x\, \upd t,
\eeq
which is invariant under translations in $x$. By Noether's Theorem, the corresponding multiplier $\mcQ=w_x$ gives the conservation law
\beq\label{pstrans}
D_t\{-w_x^2-\tfrac{1}{2}(D_x\Psi(w_x))^2\}+D_x\{w_x\Psi'(w_x)D_xD_t\Psi(w_x)\}=0
\eeq

There are two families of multipliers that depend on arbitrary functions of $t$. The first is described in Example \ref{expot} and yields the potential version of (\ref{jrk6.1}), namely
\beq\label{jrk6.6}
w_{t}= \Psi'(w_x) D_xD_t\Psi(w_x).
\eeq
The second family of multipliers, $\mcQ=\gtt(t)w_t$, are characteristics for the variational symmetries 
\beq\label{transt}
t\mapsto h(t),\qquad h'(t)\neq 0,
\eeq
which are of course also symmetries of \eqref{jrk6.1}.
The conservation law arising from \eqref{clcon} is
\beq
D_x\left\{\gtt(t)\left(-\tfrac{1}{2}w_t^2+w_t\Psi'(w_x)D_xD_t\Psi(w_x)-\tfrac{1}{2}(D_t\Psi(w_x))^2\right)\right\}=0,
\eeq
which gives the first integral
\beq
-w_t^2+2w_t\Psi'(w_x)D_xD_t\Psi(w_x)-(D_t\Psi(w_x))^2=f(t).
\eeq
Taking \eqref{jrk6.6} into account simplifies this to
\beq\label{psred2}
w_t^2-(D_t\Psi(w_x))^2=f(t).
\eeq
The variational symmetries \eqref{transt} map \eqref{psred2} to
\beq
w_t^2-(D_t\Psi(w_x))^2=f(t)(h'(t))^{-2},
\eeq
splitting the first integral into three equivalence classes. For simple representatives of each class, we choose $f(t)\in\{0, \pm 1\}$ in \eqref{psred2}. The class with $f(t)=0$ consists of two cases:
\beq
D_t\Psi(w_x)\pm w_t=0,
\eeq
each of which is consistent with \eqref{jrk6.6} and admits the above variational symmetries. Consequently, this class reduces to a first-order ODE:
\beq\label{psvar0}
\Psi(w_x)\pm w=g(x),
\eeq
where $g(x)$ is arbitrary.

Reinstating $u$ as the dependent variable, (\ref{psred2}) with a nonzero right-hand side leads to two possibilities:
\[
u_t=\pm
D_x\!\left(\left\{\left(D_t\Psi\left(u\right)\right)^2 +f(t)\right\}^\frac{1}{2}\right),\qquad f(t)\neq 0.
\]
From \eqref{psvar0}, the remaining possibilities are
\[
D_x\Psi(u)\pm u=g'(x).
\]

We now turn to the general case of (\ref{jrk6.1}), in which (\ref{jrk6.2}) need not hold. A standard calculation (as in Anco \cite{anco}) shows that, for arbitrary $M$ and $\Psi$, all multipliers that depend on $(x,t,u)$ only are linear combinations of
\begin{equation}\label{pschar1}
	\mcQ_1=1,\qquad \mcQ_2=\int\frac{\Psi'(u)}{M(u)}\,\upd u.
\end{equation}
The corresponding conservation laws are $\mcA=0$ and
\begin{equation}\label{pscl2}
	D_t\{-\Phi_2(u)-\tfrac{1}{2}(D_x\Psi(u))^2\}+D_x\{\Phi_2'(u)M(u)D_xD_t\Psi(u)\}=0,\quad\text{where}\ \Phi_2(u)=\int\mcQ_2(u)\,\upd u.
\end{equation}
This conservation law amounts to \eqref{pstrans} when the restriction \eqref{jrk6.2} holds.

There are two special cases, each of which has a family of additional multipliers depending on an arbitrary function $\gtt(x)$. If $M(u)=1$, the family is $\mcQ=\gtt(x)$, so \eqref{clcon} yields the (obvious) conservation law
\begin{equation}
D_t\{\gtt(x)(-u+D_x^2\Psi(u))\}=0.
\end{equation}
The resulting first integral,
\begin{equation}\label{psfi1}
-u+D_x^2\Psi(u)=f(x),
\end{equation}
reduces the PDE to a family of ODEs.

The second special case occurs when $M(u)=\exp(-\mu \Psi(u))$, where $\mu$ is a nonzero constant. This gives the family of multipliers $\mcQ=\gtt(x)\mu^{-1}\exp(\mu \Psi(u))$. The corresponding conservation laws \eqref{clcon} are
\begin{equation}
	D_t\left\{\gtt(x)\left(\mu^{-1}D_x^2\Psi(u)-\tfrac{1}{2}(D_x\Psi(u))^2- \Phi(u)\right)\right\}=0,\quad\text{where}\ \Phi(u)=\int\mu^{-1}\exp(\mu \Psi(u))\,\upd u.
\end{equation}
This leads to the first integral
\begin{equation}\label{psfi2}
	\mu^{-1}D_x^2\Psi(u)-\tfrac{1}{2}(D_x\Psi(u))^2-\Phi(u)=f(x).
\end{equation}
Again, a family of multipliers depending on $\gtt(x)$ has reduced the PDE to an ODE. Indeed, such a reduction always occurs when a PDE has $t$-derivatives of at most first order and multipliers that depend on an arbitrary $\gtt(x)$. Every PDE \eqref{jrk6.1} has the family of point symmetries \eqref{transt}. However, these are trivial symmetries of the reduced ODEs \eqref{psfi1} and \eqref{psfi2}, so they do not provide any further simplification.

The doubly exceptional instances of these special cases, in the sense that (\ref{jrk6.2}) also applies, are $M(u)=1$, which gives the linear PDE
\[
u_t=u_{xxt}\,,\]
and
\[
\Psi=\ln (\mu u) /\mu, \qquad \Phi=u^2 /2,
\]
so that $M(u)=\exp(-\mu \Psi(u))=1 /(\mu u)$. In the latter case, setting $u=1 / \sigma^{2}$ in \eqref{psfi2} leads to Pinney's equation in the form
\[
\sigma_{xx}+\frac{\mu^{2} f(x)}{2}\,\sigma+\frac{\mu^{2}}{4 \sigma^{3}}=0,
\]
implying the integrability of the PDE. Special cases arise from (\ref{psfi2}) much more generally, in fact. Setting
\[
\Psi (u)= -\frac{2}{\mu}\ln \sigma
\]
gives
\[
\sigma_{xx}+\frac{\mu^{2}f(x)}{2}\,\sigma+\frac{\mu^{2}}{2}\,\sigma F (\sigma)=0
\]
where $F(\sigma)=\Phi(u)$, so a second integrable case arises for $F(\sigma)=1 / \sigma$, whereby, setting $\mu =-1$,
\[
u_t= 2 D_x\!\left(u^{2}D_xD_t\ln u\right)
\]
leads via $\sigma=u$ to
\[
u_{xx}+\frac{f (x)}{2}\,u=-\,\frac{1}{2}\,.
\]

That (\ref{transt}) represents a symmetry of (\ref{jrk6.1}) and (\ref{jrk6.6}) for any non-constant $h(t)$ allows a reduction of order in general, as in the familiar ODE theory. Since $x, u, w$ and $p:= u_x$ are each invariant under (\ref{transt}), the upshot is that the third-order PDE (\ref{jrk6.1}) can be reduced to a second-order one (in divergence form) for $p(u,x)$, namely
\beq\label{jrk6.10}
D_x\!\left\{ M(u)D_u\left(p\Psi'(u) \right)\right\}+
D_u\!\left\{p M(u)D_u\! \left(p\Psi'(u) \right)\right\}=1,
\eeq
while for (\ref{jrk6.6}) it follows that $u(w,x)$ satisfies
\beq\label{jrk6.11}
\Psi'(u)D_w\!\left\{D_x \Psi(u)+uD_w\Psi(u) \right\}=1.
\eeq
That these two expressions are equivalent when $M(u)=\Psi'(u)$ follows on defining $\phi (u,x)$ by
\beq\label{jrk6.12}
\phi_u=\Psi'(u)\,D_u\!\left(p\Psi'(u)\right),
\eeq
whereby integration of (\ref{jrk6.10}) with respect to $u$ gives
\beq\label{phidef}
\phi_x + p\,\phi_u=u,
\eeq
without loss of generality. The left-hand side of \eqref{phidef} is $\phi_x$ at fixed $t$, so we can set $\phi=w$ and (\ref{jrk6.11}) and (\ref{jrk6.12}) are then equivalent. The above reductions proceed by eliminating $t$, the dependence upon which is reinstated as the function of integration resulting on solving
\[
u_x=p(u,x) \quad\text{or}\quad w_x=u(w,x),
\]
thereby reconstituting the third-order nature of the original PDE.

We conclude here with some comments about travelling-wave solutions
\[
u=u(z), \quad z=x-s(t),
\]
for which the corresponding special cases of the above results can be obtained from elementary ODE considerations: (\ref{jrk6.1}) becomes
\beq\label{jrk6.13}
\frac{du}{dz}=\frac{d}{dz}\left(M(u)\frac{d^{2}}{dz^{2}}\Psi(u)\right)
\eeq
so that
\beq\label{jrk6.14}
1=\frac{d}{du}\left(p\, M (u) \frac{d}{du}\left(p\,\Psi'(u) \right)\right),
\eeq
the special case of (\ref{jrk6.10}) that arises when $p=p (u)$. The obvious first integrals of (\ref{jrk6.13}) and (\ref{jrk6.14}) are also equivalent, and when (\ref{jrk6.2}) holds a further integration in the form
\[
u^2 + \alpha u + \beta = \left(p\,\Psi'(u) \right)^{2},
\]
for constant $\alpha$ and $\beta$ is also immediate, corresponding to the appropriate special case of (\ref{psred2}). As $p=\upd u/\upd z$, this case reduces completely to quadrature:
\[
x-s(t)=c\pm\int \frac{\Psi(u)}{\sqrt{u^2+\alpha u+\beta}}\,\upd u.
\]

\section{Further examples}

\bex
Every generalized Liouville equation in the hierarchy
\beq
\mcA^{(k)}:=D_x^{2k}u_{xy}-e^u=0,\qquad k\geq 0,
\eeq
has a variational formulation, with the Lagrangian
\[
L=\tfrac{1}{2}(D_x^{2k}u)u_{xy}-e^u.
\]
Such equations admit the variational symmetries
\beq\label{genLiousym}
\hat{x}=x,\qquad \hat{y}=h(y) \qquad \hat{u}=u-\ln\left(h^{'}(y)\right),\qquad h'(y)\neq 0, 
\eeq
and the corresponding multiplier is $\mcQ=g'(y)+g(y)u_y$. Applying Theorem \ref{bridge1} gives the first integral (for $k\geq 1$)
\beq\label{Lioufi}
\lambda= -D_x^{2k}u_{yy}+u_yD_x^{2k}u_y-\cdots+(-1)^{k-1}(D_x^{k-1}u_y)(D_x^{k+1}u_y)+\tfrac{1}{2}(-1)^k(D_x^k u_y)^2=f(y).
\eeq
For $k=0$, the first integral is \beq\label{Liou0fi}
\lambda=-u_{yy}+\tfrac{1}{2}u_y^2=f(y).
\eeq
For each $k$, the variational symmetries \eqref{genLiousym} map $\lambda$ to $(h'(y))^{-2}\lambda$, so take $f(y)\in\{0,\pm 1\}$ as representatives of each equivalence class. In particular, when $f(y)=0$, these symmetries give a reduction of order (with respect to $x$) by group splitting. To see this in action, we examine what happens for $k=0$ and $k=1$.  

The lowest-order invariants of the symmetries are
\[
x,\qquad p=u_x,\qquad q=u_{xy}e^{-u}\qquad r=u_{xx}.
\]
We derive a system of reduced equations for $q(x,p)$ and $r(x,p)$. The relationship between $q$ and $r$ is expressed by the first-order integrability condition
\beq\label{Liouint}
qr_p=u_{xxy}e^{-u}=q_x+rq_p+pq.
\eeq
The case $k=0$ (the Liouville equation) amounts to $q=1$, so the integrability condition has the solution
\[
r=\tfrac{1}{2}p^2+F(x),
\]
where $F$ is arbitrary. Together with the first integral \eqref{Liou0fi}, this gives the well-known reduction of the Liouville equation to a pair of ODEs.

More interestingly, the symmetry reduction for the case $k=1$ gives the ODE
\beq\label{Liou3}
qr_{pp}+q_pr_p+\tfrac{1}{2}q=0,
\eeq
coupled with the integrability condition \eqref{Liouint}.
This is supplemented by the first integral \eqref{Lioufi}, namely
\[
\lambda=-u_{xxyy}+u_yu_{xxy}-\tfrac{1}{2}u_{xy}^2=f(y).
\]
\eex

\bex
Consider the following time-independent curvature equation, which is relevant to capillary surfaces, for example (see King \textit{et al.} \cite{kingocken}):
\beq\label{jrk7.12}
\mcA:=\kappa(x,y)-\nabla\cdot\left(\frac{\nabla u}{|\nabla u|}\right)=0.
\eeq
This expresses the curvature of plane curves of constant $u$ (i.e. level sets), in terms of a given function $\kappa(x,y)$. Equation (\ref{jrk7.12}) is the Euler--Lagrange equation corresponding to the functional
\beq\label{jrk7.13}
\mathcal{L}=\int\hspace{-0.2cm}\int\left(|\nabla u|+\kappa(x,y)u\right)\,\upd x\,\upd y.
\eeq
While (\ref{jrk7.12}) is evidently invariant under $u\rightarrow U(u),\ U'(u)\neq 0$, these symmetries are not variational.

In the special case where $\kappa$ is independent of $x$,
there exists a family of multipliers, $\mcQ=\gtt(u)u_x$. These are nontrivial in the region where $u_x\neq 0$, so a conservation law-reduction can be obtained by a hodograph transformation, treating $u$ as an independent variable and $x$ as the dependent variable. To transform \eqref{jrk7.12}, without restricting $\kappa$ in advance, let $\phi(u,y)=\kappa(x(u,y),y)$ to obtain
the following family of conservation laws that hold, remarkably, for \emph{all} $\kappa(x,y)$:
\beq
D_y\left\{\gtt(u)\left(\frac{x_y}{(1+x_y^2)^{1/2}}+\Phi(u.y)\right)\right\}=0,\quad\text{where}\quad \Phi(u,y)=\int\phi(u,y)\,\upd y.
\eeq
This leads to the first integral
\[
\frac{x_y}{(1+x_y^2)^{1/2}}+\Phi(u.y)=f_1(u),
\]
from which \eqref{jrk7.12} can be solved by quadrature:
\[
x=f_2(u)\pm\int\frac{f_1(u)-\Phi(u,y)}{\{1-(f_1(u)-\Phi(u,y))^2\}^{1/2}}\,\upd y.
\]
Here $f_1$ and $f_2$ are arbitrary, subject to the constraint $|f_1(u)-\Phi(u,y)|< 1$ (which amounts to $u_x$ being real and nonzero). A similar reduction, with $y=y(x,u)$, is applicable in regions where $u_y\neq 0$.
\eex

\section{Concluding remarks}
We have shown that conservation law multipliers that depend on functions of the independent variables can be used to simplify or solve PDEs. Reduced Euler--Lagrange systems generally inherit equivalence transformations from the multipliers, though a subset of solutions may be invariant, leading to a second reduction by the inherited symmetries.

Where there are two independent variables, the reduced PDE yields a first integral. More generally, if there are $p>2$ independent variables, reduction using multipliers that have arbitrary dependence on $s<p$ independent variables leads to a conservation law with $p-s$ components (or a first integral when $s=p-1$). Hodograph transformations enable $u$ to be used as an independent variable.

For variational problems, multipliers are characteristics of variational symmetries, from which Noether's theorems follow. Indeed, in the variational case, Theorem \ref{bridge1} reduces to a theorem in Hydon \& Mansfield \cite{hydman} that bridges the gap between Noether's theorems (and so has become known informally as Noether $1\tfrac{1}{2}$).

Throughout, we have imposed the restriction that arbitrary functions in multipliers depend only on whichever variables are being regarded as independent.
However, a nontrivial conservation law is zero only on solutions of the PDE, which are graphs $\mbx\mapsto (\mbx,[\mbu(\mbx)])$. So each $\gtt^r$ can depend on $(\mbx,[\mbu])$, subject to the system of constraints \eqref{constr} being satisfied for all solutions. (Where the arbitrary functions are unconstrained, this observation amounts to the substitution principle of Kiselev \cite{kiselev}, which is used in Olver \cite{olvN2} to prove that the variational symmetry characteristics associated with Noether's Second Theorem can depend on arbitrary functions of $(\mbx,[\mbu])$.) Nevertheless, the calculations are considerably easier if one restricts to functions $\gtt^r(\mbx)$, as Noether did.

\section*{Acknowledgements} We are grateful to the organizers of the 59th British Applied Mathematics Colloquium (BAMC), which led to the authors' collaboration in this research. The second author gratefully acknowledges a Leverhulme Trust Fellowship.
\medskip

\noindent Competing interests: The authors declare none.

\end{document}